\documentclass[a4paper,7pt,oneside,onecolumn,number,preprint,centertitle]{elsarticle}

\usepackage{amsmath,amssymb,bm}
\usepackage{graphicx}
\usepackage{textcomp}
\usepackage{xcolor}
\usepackage{amsthm}
\usepackage{tikz}

\usepackage{enumerate}
\usepackage{bbold}
\usepackage{epsfig}
\usepackage{multicol}
\usepackage{float}
\usepackage{hyperref}
\usepackage{multirow}
\usepackage{mathrsfs} %
\usepackage{dsfont}
\usepackage{algorithm}
\usepackage{algorithmicx}
\usepackage{algpseudocode}
\usepackage{flushend}

\usepackage{ifthen}
\DeclareMathAlphabet{\mbb}{U}{bbold}{m}{n} %
\newcommand \mb[1] {\ifthenelse{\equal{#1}{0}}{\mbb{0}}{\ifthenelse{\equal{#1}{1}}{\mbb{1}}{\mathbb{#1}}}} %

\renewcommand \subset {\subseteq}

\renewcommand{\l}{\left}

\newtheorem{thm}{Theorem}

\newtheorem{rem}{Remark}
\newtheorem{lem}[thm]{Lemma}
\newtheorem{ass}{Assumption}
\newtheorem{defn}{Definition}

\newtheorem{exmp}{Example}%

\renewcommand{\textcolor}[2]{#2}
\begin{document}
\begin{frontmatter}

\title{$L_1$ Optimal Control of Continuous-Time Stochastic Positive Systems}

\author[inst1]{Alba Gurpegui}
    \ead{alba.gurpegui_ramon@control.lth.se}

\affiliation[inst1]{organization={Department of Automatic Control, Lund University},%
            addressline={Ole Römers väg 1}, 
            city={Lund},
            postcode={223 63}, 
            country={Sweden}}

\author[inst2]{Takashi Tanaka}
    \ead{tanaka16@purdue.edu}

    \affiliation[inst2]{organization={School of Aeronautics and Astronautics, Elmore Family School of Electrical and Computer Engineering, Purdue University},%
            addressline={701 W. Stadium Ave.}, 
            city={West Lafayette},
            postcode={IN 47907-2045}, 
            country={United States}}

\author[inst1]{Anders Rantzer}
    \ead{anders.rantzer@control.lth.se}

\begin{abstract}
We present an $L_1$-optimal control problem class with linear nonnegative costs subject to multiplicative Itô diffusion processes with elementwise linear input constraints. Forward invariance of the positive orthant is established for the considered stochastic dynamics, and a simulation method consistent with this invariance property is proposed. Both finite-horizon and discounted infinite-horizon stochastic $L_1$-optimal control problems are considered. These problems admit explicit solutions characterized by a vector-valued ordinary differential equation in the finite-horizon case and by an algebraic equation in the infinite-horizon case. Notably, the optimal value function and feedback policy coincide with those of the corresponding deterministic problem, demonstrating robustness to multiplicative stochastic uncertainty. A portfolio example illustrates our results.
\footnotetext{
This work is partially funded by the Wallenberg AI, Autonomous Systems and Software Program (WASP), the European Research Council under the  grant agreement No 101199738, DARPA grant HR0011-25-3-0210 and AFOSR grant FA9550-25-1-0347.}
\end{abstract}


\end{frontmatter}

Stochastic optimal control is a fundamental area of control theory with several engineering applications~\cite{Astrom_stoch_control}. A classical example is the Linear Quadratic Gaussian problem (LQG)~\cite{kalman_lqg}, which minimizes a quadratic cost for linear systems with Gaussian noise and admits closed-form solutions. While $L_2$-optimal control of stochastic systems has proven to be effective in a wide range of applications, including robotics, aerospace, and process control~\cite{Robotics_L2, aerospace_LQG, LQG_process_cont}, the $L_1$-norm is often a more appropriate choice in the presence of actuation constraints and sparse resource allocation~\cite{sparse_L1}, for instance in minimum-fuel control problems~\cite{min_fuel} or maximum hands-off control~\cite{max_hands_off}.  
Recent work on deterministic $L_1$-optimal control has introduced a novel optimal control problem class that admits explicit solutions in discrete~\cite{yuchao} and continuous time~\cite{LR_paperII}. We refer to this framework as the Linear Regulator problem (LR), which is characterized by linear costs, positive linear dynamics and elementwise linear constraints on inputs.

A relevant feature of the LR framework is the positivity of the dynamics. In deterministic continuous-time positive systems, forward invariance of the positive orthant follows directly from the Metzler structure of the state matrix~\cite{ Luenberger, FarinaRinaldi2000}. {Beyond this classical setting, the analysis and control of positive systems have been studied for singular and impulsive dynamics with time delays~\cite{Z1,Z2,Z3}}. In contrast, for Itô diffusion processes the diffusion term may drive trajectories outside the positive orthant even when the drift term in the equation satisfies the corresponding positivity conditions. Classical results on stochastic invariance and viability are given in~\cite{Da_prato_SDE_invariance, Milian_invariance_sde}, { including a stochastic version of the Nagumo viability theorem for stochastic differential equations with Lipschitz or monotone dynamics}~\cite{nagumo_stochastic}. Related questions regarding positivity, stability~\cite{stab_markov} and the $L_{\infty}$ and $L_1$ performance~\cite{L1_markov, infty_markov} have also been studied for stochastic systems with Markov and semi-Markov jump processes. In this paper we present a stochastic extension of the LR framework, subject to {a class of linear Itô diffusions with multiplicative
noise and bounded state-scaled control inputs.} The main contributions of this paper are summarized as follows.
\begin{enumerate}[C1]
    \item We prove forward invariance of the positive orthant for this class of Itô diffusions with bounded inputs {and propose a positivity-preserving time discretization algorithm.} 
    \item We propose a finite-horizon and a discounted infinite-horizon LR optimal control problem for positive systems with Gaussian multiplicative noise.
    \item We derive explicit solutions for both settings. The finite-horizon solution is given by a vector-valued ODE and a time-varying switching policy, and the discounted infinite-horizon solution is a vector-valued algebraic equation and an optimal static feedback policy.
    \item We show that the value function and optimal policy coincide with those of the deterministic LR problem, demonstrating robustness with respect to multiplicative stochastic uncertainty.
\end{enumerate}
\subsection{Notation}
{Let $\mathbb{R}_{+}$ denote the set of nonnegative real numbers. $X > Y$ $(X \geq Y)$ means all elements of the matrix $(X-Y)$ are positive (nonnegative).  A matrix $X$  is positive if its entries are nonnegative and at least one is nonzero. $\left | X \right |$ means elementwise absolute value and $\alpha(X)$ the spectral abscissa of $X$. $\mathbb E^{t,x}[\cdot]:=\mathbb E[\cdot | x(t)=x]$ is the conditional expectation. For a set $\mathcal Q \subset \mathcal A \times \mathcal B$, $C^{1,2}(\mathcal Q)$ is the class of functions that are continuously differentiable in $\mathcal A$ and twice continuously differentiable in $\mathcal B$. {The signum of a scalar, $\mathrm{sign}(x)$, equals $\big\{-1\big\}$ if $x<0$, $[-1, 1]$ if $x=0$ and $\big\{+1\big\}$ if $x>0$.}
\section{Stochastic Positive Systems}
A linear system is said to be positive if its state and output remain nonnegative for all nonnegative initial conditions and inputs. In particular, positivity of continuous-time linear systems is related to Metzler matrices.
\begin{defn}[Metzler]
    A square matrix $A$ is said to be Metzler if each off-diagonal element $a_{i,j}$, $i\neq j$ of $A$ is nonnegative.
\end{defn}
\begin{lem}\label{lem_pos}
    Consider the linear time-varying system $\dot x(t)=A(t)x(t), \hspace{1mm} t\in [0,T]$, 
    where $A(t)$ is a bounded, time-varying Metzler matrix for each $t\in [0, T]$. If $x(0)$ is entry-wise nonnegative, then $x(t)$ remains entry-wise nonnegative for all $t\in [0,T]$.
\end{lem}
\begin{proof}
    See~\cite[Sec. VIII]{Angeli}.
\end{proof}

In what follows, we assume that all the random processes
are defined on a probability space $(\Omega, \mathcal F, P)$ {~\cite{Oksendal}}. Let $\mathcal F_t \subset \mathcal F$ be a nondecreasing family of $\sigma$-algebras, and let $w(t)\in \mathbb R^N$, $t \geq 0$, $w(0)=0$, be an $\mathcal F_t$-adapted $N$-dimensional Brownian motion (BM) satisfying
\begin{align*}
    \mathbb E[(w(t)-w(s))(w(t)-w(s))^{\top}]=(t-s)I_N,
\end{align*}
for all $0\leq s \leq t$. In this paper, we are concerned with the dynamical system
defined by the following Itô stochastic differential equation
(SDE):
\begin{align}\label{SDE}
    dx(t)&=f(x(t), u(t))dt+\textstyle\sum\nolimits_{n=1}^N F_n x(t)d w_n(t) \notag\\
    x(0)&=x_0, \hspace{1mm} x_0\in \mathbb R^L_+ \hspace{2mm}0\leq t\leq T
\end{align}
where $f(x(t), u(t))=Ax(t)+\sum_{m=1}^M B_m u_m(t)$ {is the drift, $F_n \in \mathbb R^{L\times L}$ are the diffusion matrices $n=1,\dots, N$,}  $A, B_m \in \mathbb R^{L\times L}$ for each $m=1,\dots,M$, and $x(t)\in \mathbb R^L $, $u_m(t)\in \mathbb R^L$. The vector $x(t)$ denotes the time-varying state, {the diffusion term in~\eqref{SDE} is a multivariate geometric Brownian motion (GBM)~\cite{Hu_finance}} and $u_m(t)= {K_m}(t)x(t)$ are the state scaled, bounded control inputs, {$K_m(t)$ is a diagonal $L\times L$ matrix} {with $|{K_m(t)}|\leq I_L e_m $, where $I_L$ is the $L$-dimensional identity matrix} and $e_m\in \mathbb R_+$ for $m=1,\dots, M$. In the sequel, we interpret $x(t)$ as the strong solution to the SDE~\eqref{SDE} whose existence and uniqueness {are guaranteed by requiring that an admissible control input $u(\cdot)$ is such that the drift and diffusion terms of the SDE~\eqref{SDE} satisfy the conditions of~\cite[Def. 3.1.4]{Oksendal} and~\cite[Thm. 5.2.1]{Oksendal}.} 
\begin{ass}\label{ass_1}
Suppose that the dynamical system~\eqref{SDE} has the following additional properties:
    \begin{enumerate}[(a)]
    \item $F_n$ is a diagonal matrix for each $n$;
    \item $A-\textstyle\sum\nolimits_{m=1}^M |B_m|e_m $ Metzler.
\end{enumerate}
\end{ass}
\begin{rem}
    {Condition $(b)$ of} Assumption~\ref{ass_1} ensures the invariance of the positive orthant for  the drift term of~\eqref{SDE}.
\end{rem}
 {The positivity of the solution of multi-dimensional GBM with deterministic drift and diffusion coefficients was investigated in~\cite{Hu_finance}, where forward invariance of the positive orthant is established in the absence of control inputs. The following theorem extends this} to the controlled Itô diffusion considered here and shows that the positive orthant is forward invariant almost surely i.e. if $x(0)=x_0\in\mathbb R^L_+$ then $x(t)\in \mathbb R^L_+$ for all $t\geq 0$ and for all admissible $u(t)$. This property can alternatively be approached through the stochastic invariance theory for differential inclusions~\cite[Thm.~5.1]{Da_prato_SDE_invariance}; we give a direct constructive proof, which moreover underlies Algorithm~\ref{algorithm_GBM}. 
\begin{thm}\label{Thm1}
    Consider a dynamical system of the form~\eqref{SDE} that satisfies Assumption~\ref{ass_1}. Suppose that {$u(t)=K_m(t)x(t)$ is such that conditions of~\cite[Def. 3.1.4]{Oksendal} and~\cite[Thm. 5.2.1]{Oksendal} are satisfied} {and $|K_m(t)|\leq I_L e_m$ for $m=1,\dots, M$.} If the initial state {$x_0\in \mathbb R^L_+$,} then $x(t)$ remains component-wise nonnegative for all $t>0$.
\end{thm}
\begin{proof}
    Let $a_{ij}$ be the $(i,j)$ element of the matrix $A$. The $i$-th diagonal element of $F_n$ will be denoted by $F_{n,i}$.

    Define $f_i:= \sqrt{\textstyle\sum\nolimits_{n=1}^{N}F_{n,i}^2}$. If $f_i>0$, let 
    \begin{align*}
        v_i(t):= \textstyle\sum\nolimits_{n=1}^N\int_0^{t}\frac{F_{n,i}}{f_i}d w_n(s).
    \end{align*}
    Then $v_i(t)$ is the standard BM for each $i=1,\dots,L$.  If $f_i=0$, the diffusion term in the $i$-th equation vanishes and the positivity of $x_i(t)\geq 0$ follows from Lemma~\ref{lem_pos}.
     
    The $i$-th element of~\eqref{SDE} can be written as 
    \begin{align}
        dx_i(t)=\textstyle\sum\nolimits_{j=1}^L \bar a_{ij}(t)x_j(t)dt+f_i x_i(t)dv_i(t),
    \end{align}
    where $\bar a_{ij}(t)= a_{ij}+\textstyle\sum\nolimits_{m=1}^M [B_m]_{ij}{[K_m(t)]_{jj}}$. For $i \neq j$, Assumption~\ref{ass_1} $(b)$ implies that $a_{ij}-\sum_m |[B_m]_{ij}|e_m\geq 0$. Thus, $\bar a_{ij}(t)\geq 0$ for all $i\neq j$.  Define $y_i(t):=x_i(t)e^{-f_i v_i(t)}$. By the Itô formula~\cite{Oksendal},
    \begin{align*}
        dy_i(t)&=\frac{\partial y_i}{\partial x_i}dx_i(t)+ \frac{\partial y_i}{\partial v_i}dv_i(t)\\
        &+\frac{1}{2}\begin{bmatrix}
            dx_i(t)\\
            dv_i(t)
        \end{bmatrix}^{\top}\begin{bmatrix}
            \frac{\partial^2 y_i}{\partial x_i^2}& \frac{\partial^2 y_i}{\partial x_i\partial v_i}\\
            \frac{\partial^2 y_i}{\partial x_i\partial v_i}& \frac{\partial^2 y_i}{\partial v_i^2}
    \end{bmatrix}\begin{bmatrix}
                dx_i(t)\\
                dv_i(t)
            \end{bmatrix}\\
        &=\left(\textstyle\sum\nolimits_{j=1}^L \bar a_{ij}(t)x_j(t)-\frac{1}{2}f_i^2 x_i(t) \right)e^{-f_i v_i(t)}dt\\
        &= \left(\textstyle\sum\nolimits_{j=1}^L \bar a_{ij}(t)e^{f_j v_j(t)-f_iv_i(t)}y_j(t)-\frac{1}{2}f_i^2 y_i(t) \right)dt
    \end{align*}
    Hence $\frac{dy_i(t)}{dt}=\sum_{j=1}^L \tilde a_{ij}(t)y_j(t)$ where 
    \begin{align*}
        \tilde a_{ij}(t)=\begin{cases}
            \bar a_{ij}(t) e^{f_j v_j(t)-f_iv_i(t)}&j\neq i\\
            \bar a_{ii}(t)-\frac{1}{2}f_i^2&j=i.
        \end{cases}
    \end{align*}
    Observe that $\tilde a_{ij}(t)\geq 0$ for $j\neq i$. Hence, by Lemma~\ref{lem_pos}, if $y_i(0)\geq 0$ for each $i$, then $y_i(t)\geq 0$ for each $i$ for all $t>0$. Therefore, if $x_i(0)=y_i(0)e^{f_i v_i(0)}\geq 0$ for each $i$, then $x_i(t)=y_i(t)e^{f_i v_i(t)}\geq 0$ for each $i$ for all $t>0$.
\end{proof}

If the diffusion term in~\eqref{SDE} were additive and independent of the state, it would not vanish at the boundary, and trajectories could become negative with positive probability. The following example illustrates this phenomenon. 
\begin{figure}[h]
    \centering
    \includegraphics[scale=0.55]{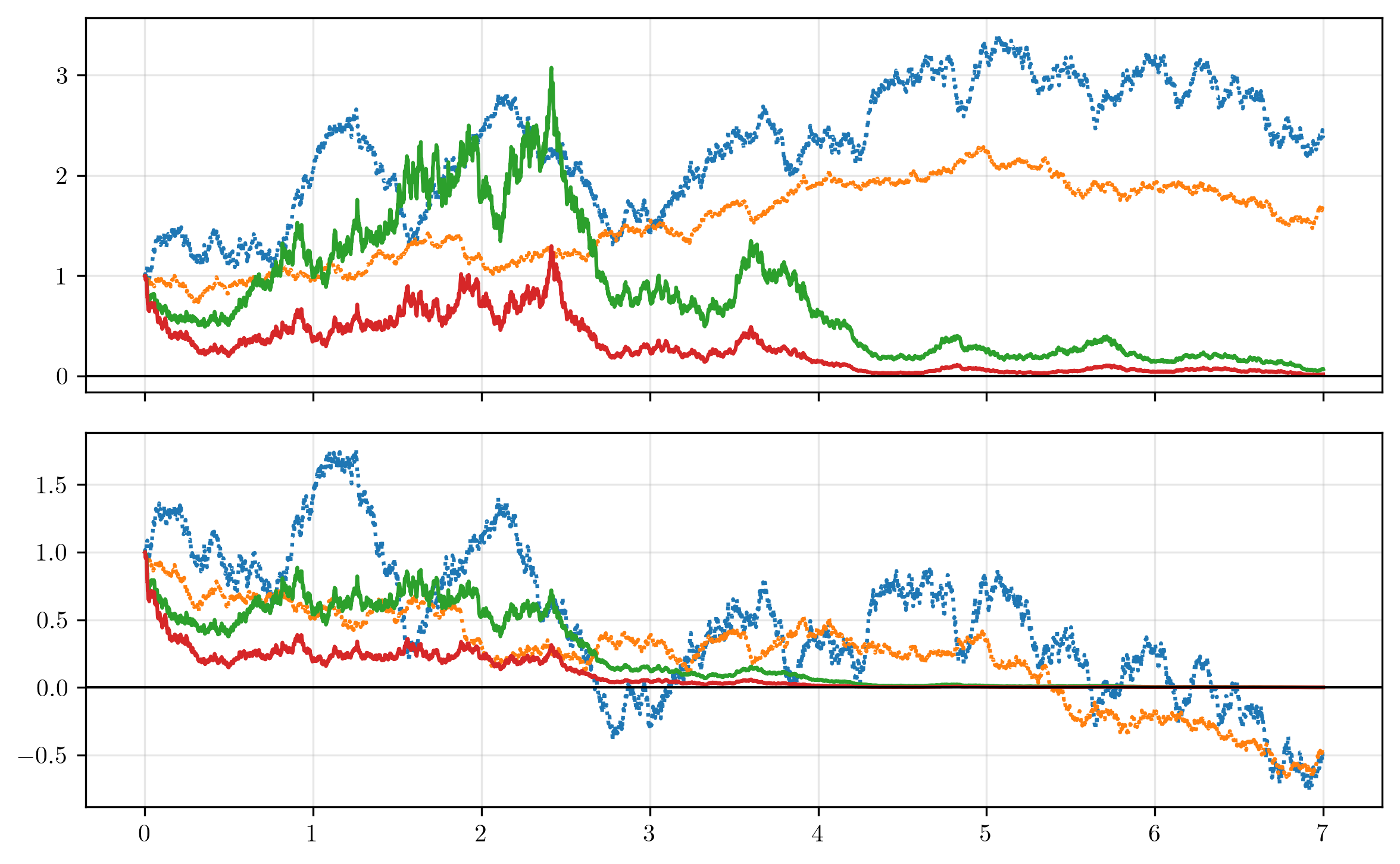}
    \caption{Time evolution of the systems in Example~\ref{ex1}. Top: $k=0$, bottom: $k=-0.6$. System $a)$ trajectories: dotted (blue, orange). System $b)$ trajectories: solid (red, green).}
    \label{exmp}
\end{figure}
\begin{exmp}\label{ex1}
    Consider the system dynamics
    \begin{align} \label{sys_ex1}
        &a) \hspace{1mm}dx_a(t)=(A+BK) x_a(t)dt+\textstyle\sum\nolimits_{n=1}^2 f_n d w_n(t)\\
        &b) \hspace{1mm}dx_b(t)=(A+BK)x_b(t)dt+\textstyle\sum\nolimits_{n=1}^2 D_n x_b(t)d w_n(t) \label{sys_ex2}
    \end{align}
    with $x_a(0)=x_b(0)=x_0 \in \mathbb R^2_+$, $A,B\in \mathbb R^{2\times 2}$, {$K \in \mathbb R^{2\times 2}$, with $K= I_2 k$, $k\in \mathbb R$ and} $D_n=\mathrm{diag}(f_n)$, $f_n\in \mathbb R^2$. Observe that system $b)$ admits a closed-form solution if the matrices ${A+BK}, D_1$ and $D_2$ commute
    \begin{align*}
        x_b(t)=x_0 \mathrm{exp}\Big(((A+BK)-\frac{1}{2}\textstyle\sum\nolimits_{n=1}^2 D_n^2)t+\textstyle\sum\nolimits_{n=1}^2 D_n w_n(t) \Big).
    \end{align*} 
    This solution is nonnegative for all $t$ whenever $x_0\geq 0$. In the non-commuting case, by Theorem~\ref{Thm1} the solution can be simulated using Algorithm~\ref{algorithm_GBM}. Figure~\ref{exmp} illustrates this for $N=2$, with $k=0$ and $k=-0.6$. The system parameters are
    \begin{align*}
      A=\begin{bmatrix}
        0.02&0.07\\
        0.05&0.06
    \end{bmatrix},\hspace{0.5mm}B=\begin{bmatrix}
        1&0.01\\
        -0.01&1
    \end{bmatrix},\hspace{0.5mm} f_1=\begin{bmatrix}
        0.7\\0.9
    \end{bmatrix},\hspace{0.5mm}   f_2=\begin{bmatrix}
        0.01\\0.3
    \end{bmatrix},
    \end{align*} 
    $x_a(0)=x_b(0)=\mathds 1$. We choose the bound $e=0.6$ so that {$|k|\leq e$ for both $k=0$ and $k=-0.6$.} Since $A-|B|I_2 e$ is Metzler, $A+BK$ is Metzler for both inputs. System $a)$ is simulated using the Euler-Maruyama method and System $b)$ using Algorithm~\ref{algorithm_GBM} because ${A+BK}, D_1, D_2$ do not commute.
\end{exmp}
\subsection{Positivity Preserving Simulation}
Although the SDE~\eqref{SDE} preserves the positive orthant under the assumptions of Theorem~\ref{Thm1}, standard discretization schemes generally do not. For example, the Euler–Maruyama method may produce negative iterates even for scalar GBM (Example~\ref{exmp_EM}).

\begin{algorithm}[H]
\caption{Positivity-preserving simulation for~\eqref{SDE}}
\label{algorithm_GBM}
\begin{algorithmic}[1]
\State \textbf{Input:} $T>0$, $N_t\in\mathbb N$, $x_0\in\mathbb R^L_+$;
$A=[a_{ij}]$; $B_m$ ($m=1{:}M$); $F_n$ ($n=1{:}N$) diagonal;
{diagonal matrices $K_m(t)\in \mathbb R^{L\times L}$ with $|K_m(t)|\le I_L e_m$ for $m=1:M$};
\textbf{assume} $A-\sum_{m=1}^M e_m|B_m|$ is Metzler.
\State $\Delta t\gets T/N_t$, $t_k\gets k\Delta t$; \quad $v(t_0)\gets 0$, $y(t_0)\gets x_0$
\State $f_i\gets\sqrt{\sum_{n=1}^N F_{n,i}^2}$ for $i=1{:}L$
\For{$k=0,\dots,N_t-1$}
    \State Sample $\xi_{n,k}\sim\mathcal N(0,1)$; set $\Delta w_{n,k}\gets\sqrt{\Delta t}\,\xi_{n,k}$ for $n=1{:}N$
    \For{$i=1{:}L$}
        \If{$f_i=0$}
            \State $v_i(t_{k+1})\gets v_i(t_k)$
        \Else
            \State $v_i(t_{k+1})\gets v_i(t_k)+\sum_{n=1}^N \frac{F_{n,i}}{f_i}\Delta w_{n,k}$
        \EndIf
    \EndFor
    \State Form $\bar A(t_k)\gets A+\sum_{m=1}^M B_m\,K_m(t_k)$
    \State $\tilde a_{ij}(t_k)\gets \bar a_{ij}(t_k)\exp(f_jv_j(t_k)-f_iv_i(t_k))$ for $i\neq j$
    \State $\tilde a_{ii}(t_k)\gets \bar a_{ii}(t_k)-\tfrac12 f_i^2$
    \State \scalebox{1}{$\widetilde A_k\gets[\tilde a_{ij}(t_k)]_{i,j=1}^L$}; \quad 
           \scalebox{1}{$y(t_{k+1})\gets \exp(\widetilde A_k\Delta t)\,y(t_k)$}
\EndFor
\State $x_i(t_k)\gets y_i(t_k)\exp(f_iv_i(t_k))$ for $i=1{:}L$, $k=0{:}N_t$
\State \textbf{Output:} $\{x(t_k)\}_{k=0}^{N_t}$
\end{algorithmic}
\end{algorithm}

In the scalar case, positivity can be preserved by exploiting the closed-form representation of the solution, which enables simulation by generating the Gaussian exponent and exponentiating it. Other approaches for preserving positivity in the scalar case   
include logarithmic transformations~\cite{Highman}. However, analytical solutions are generally unavailable for multivariable systems with multiplicative noise, making this extension nontrivial.  The proposed method, Algorithm~\ref{algorithm_GBM}, generalizes this idea and preserves the forward invariance established in Theorem~\ref{Thm1}.
\begin{exmp}\label{exmp_EM}
    Consider the scalar GBM $dx(t)=\alpha x(t)dt+ \sigma x(t)dw(t), \hspace{1mm} x(0)>0$. Its exact solution is $x(t)=x(0)\operatorname{exp}\left((\alpha-0.5\sigma^2)t+\sigma w(t) \right)$, which is strictly positive almost surely for all $t\geq 0$. However, Euler–Maruyama discretization gives $x_{k+1}=x_k+\alpha x_k\Delta t+\sigma x_k \Delta w_k$. Since $\Delta w_k \sim \mathcal N(0, \Delta t)$ is unbounded from below, the factor $(1+\alpha \Delta t + \sigma \Delta w_k)$ can be negative with positive probability. Thus, the discrete-time approximation may leave the positive orthant even though the continuous-time solution does not.
\end{exmp}
\section{$L_1$ OPTIMAL CONTROL PROBLEM}
In this section, we present the stochastic optimal control
problem central to this paper. Define the set of all admissible time-state pairs as $\mathcal{Q}=[0, T] \times \mathbb R^L_+$. The class of admissible control policies $\mathcal U(x{(\cdot)})$ is defined by
\begin{align}\label{admissible_set}
    \mathcal U(x{(\cdot)})= \left\{ u(\cdot): \hspace{0.1mm} \begin{matrix}
        &|u_m(t)|\leq e_m x(t), \hspace{0.1mm} \forall \hspace{0.11mm} m,t
    \end{matrix} \right\},
\end{align}
with $e_m \in \mathbb R_+$ nonnegative scalar for all $m=1,\dots, M$, {where, again $u(\cdot)$ is such that  conditions of~\cite[Def. 3.1.4]{Oksendal} and~\cite[Thm. 5.2.1]{Oksendal} are satisfied.} We are concerned with the following stochastic optimal
control problem with zero terminal cost:
\begin{subequations}\label{optprob}
\begin{align}
    &\min_{u \in \mathcal U(x{(\cdot)})}\mathbb E \Big[ \int_0^{T} \big( q^{\top}x(t)+\textstyle\sum\nolimits_{m=1}^M r^{\top}_m u_m(t) \big) dt\Big]\\
    &\hspace{2mm}\text{s.t.}\hspace{1mm} \text{Equation}~\eqref{SDE}, \hspace{1mm} \hspace{1mm} x_0\in \mathbb R^L_+, \hspace{1mm} u(t)\in \mathcal U(x{(\cdot)}),
\end{align}
\end{subequations}
with $q \in \mathbb R_+^L$ and $r_m \in \mathbb R^{L}$ for $m=1,\dots, M$. 

We propose a dynamic programming approach to characterize the optimal control policy for~\eqref{SDE}. For each $(t,x)\in \mathcal Q$ and $u \in \mathcal U(x{(\cdot)})$, define the cost-to-go function by
\begin{align}\label{cost}
    C(t,x,u)= \mathbb E \Big[\int_t^T  \big( q^{\top}x(s)+\textstyle\sum\nolimits_{m=1}^M r^{\top}_m u_m(s) \big) ds\Big].
\end{align}
We also define the value function $J: \mathcal Q \rightarrow \mathbb R_+$ by $$J(t,x)=\inf_{u\in \mathcal U(x{(\cdot)})} C(t,x,u).$$

The next theorem is the main result of this section, which summarizes the existence of an optimal control policy and its 
characterization in terms of the stochastic Hamilton-Jacobi-Bellman equation:
\begin{thm}\label{thm_finite}
    Let the matrices $A, B_m, F_n \in \mathbb R^{L\times L}$ satisfy the properties in Assumption~\ref{ass_1}, and assume that
    $q \in \mathbb R^L_+$, and $r_m\in \mathbb R^L$, $m=1,\dots M$. Suppose that \begin{align}\label{ass_nonneg}
        q> \textstyle\sum\nolimits_{m=1}^{M}|r_m|e_m.
    \end{align}
    Let $p:[0,T]\rightarrow \mathbb R^L_+$ be the unique solution to the backward ODE
    \begin{align}\label{ODE}
            -\dot{p}(t)&=q+A^{\top}p(t)-\textstyle\sum\nolimits_{m=1}^M e_m|r_m+B^{\top}_m \hspace{0.5mm}p(t)|\\
            p(T)&=0.\notag
        \end{align}
    Then the following hold: 
    \begin{enumerate}[(i)]
        \item  $J(t,x)=p^{\top}(t)x$ is the value function for the optimal control problem~\eqref{optprob}.
        \item Any optimal control is given by {$u^*_m(t)=-K_m(t)x(t)$ where $K_m$ satisfies~\eqref{opt_u_fin} for almost every $t\in [0,T]$.}
    \begin{align}\label{opt_u_fin}
       {K_m(t)\in}\; \mathrm{diag}(\mathrm{sign}(r_m^{\top}+p^{\top}(t)B_m))e_m.
    \end{align}  
    \end{enumerate}
\end{thm}
\begin{rem}
    Condition~\eqref{ass_nonneg} ensures that the cost function is bounded from below.
\end{rem}
\begin{rem}
    The right-hand side of~\eqref{opt_u_fin} is set-valued at indices $j$ where $(r_m)_j+(p(t)^{\top}B_m)_j=0$. In these cases $[K_m(t)]_{jj}$ may take any value in $[-e_m, e_m ]$. This is a singular arc, as the running cost is linear, hence not strictly convex in the control variable. The Hamiltonian is affine in $u_m$, and therefore the first-order optimality conditions do not uniquely determine the optimal control. Nonetheless all solutions yield the same $p(t)$ of~\eqref{ODE}.
\end{rem}
\begin{rem}\label{rem4}
    The policy~\eqref{opt_u_fin} coincides with the optimal control law of the deterministic setting~\cite[Thm. 1]{LR_paperII}, {demonstrating that the optimal controller is robust to the multiplicative stochastic uncertainty in~\eqref{SDE}. This equivalence is a structural consequence of the linearity of the value function candidate and problem setting, which makes the second-order term in the HJB equation vanish. This deterministic-equivalence interpretation is conceptually related to the certainty-equivalence phenomenon established for discrete-time semilinear dynamic programming in~\cite[Sec. 5]{semilinear}, where the value iteration and optimal policy for certain stochastic problems coincide with those of the corresponding deterministic problem obtained by replacing random parameters with their expectations.}
\end{rem}

\begin{rem}
    An equivalent maximization problem (see~\cite[Thm. 1]{LR_paperII} with $u,w=0$) can be obtained by changing the sign of the running cost,  
    \begin{align}\label{max_obj}
        \max_{u \in \mathcal U(x{(\cdot)})}\mathbb E \Big[ \int_0^{T} \big(q^{\top}x(t)- \textstyle\sum\nolimits_{m=1}^M r_m^{\top} u_m(t) \big)dt \Big].
    \end{align}
    Assuming $q> -\textstyle\sum\nolimits_{m=1}^{M}|r_m|e_m$, the value function remains $J(t,x)=p^{\top}(t)x$, and the vector $p(t)$ satisfies 
    \begin{align*}
        -\dot{p}(t)=q+A^{\top}p(t)+\textstyle\sum\nolimits_{m=1}^M e_m|-r_m+B^{\top}_m p(t) |
    \end{align*}
    where $u^*_m(t)=-K_m(t)x(t)$, with $K_m(t) \in \mathrm{diag}(\mathrm{sign}(-r_m^{\top}+p^{\top}(t)B_m))e_m.$
\end{rem}
\begin{proof} We prove Theorem~\ref{thm_finite}. {By the Picard–Lindelöf theorem~\cite{PicardLind}, the ODE~\eqref{ODE} $p:[0,T]\rightarrow \mathbb R_+^L$ admits a unique solution and nonnegativity of $p(t)$ on $[0,T]$  follows from the monotone-systems argument of~\cite[Lem. 2]{LR_paperII}.} Define $J(t,x)=p^{\top}(t)x$.  Since $J(t,x)\in C^{1,2}(\mathcal Q)$, the verification theorem~\cite[Thm. 8.1]{viscosity_fleming} applies. By Dynkin's formula for the graph of Itô diffusion~\cite[(2.7)]{viscosity_fleming}
    \begin{align*}
        &\mathbb E^{t,x}[J(T, x(T))]-J(t,x)=\mathbb{E}^{t,x}\Big[\int_t^T \mathcal A^{u} J(s, x(s))ds \Big],
\end{align*} 
where the backward evolution operator $\mathcal A^{u}$ applied to $J$ is
\begin{align}\label{operator_A}
    \mathcal A^{u}J
    &= \partial_s J 
    + (Ax + \textstyle\sum\nolimits_{m=1}^M B_m u_m)^{\top}\partial_x J \\
    &~~~~~~~~~~~~~~~~~~~~~\quad + \tfrac12 \textstyle\sum\nolimits_{n=1}^N 
    \mathrm{Tr}\hspace{0.1mm}\big(F_n x x^{\top} F_n^{\top}\partial_x^2 J\big). \notag
\end{align}
    Substituting $J(t,x)=p(t)^{\top}x$, we note that the second-order term vanishes ($\partial^2_x J=0$) and the stochastic integral has zero expectation. This gives
    \begin{align}\label{18}
    &J(t,x)=
    -\mathbb{E}^{t,x}\!
    \Big[\int_{t}^{T}
    \big[
    \dot p(s)^{\top}x(s)
    \\
    &~~~~~~~~~\quad
    +p(s)^{\top}\big(Ax(s)+\textstyle\sum\nolimits_{m=1}^M B_m u_m(s))
    \big]
    \, ds\Big].\notag
    \end{align}
    Observe that the right-hand side of the ODE~\eqref{ODE} can be expressed as 
    \begin{align*}
    -\dot p^{\top}x
    = \min_{u}
    \Big(
    q^{\top}x 
    + \textstyle\sum\nolimits_{m=1}^M r_m^{\top}u_m
    + p^{\top}(Ax + \textstyle\sum\nolimits_{m=1}^M B_mu_m)
    \Big).
    \end{align*}
Therefore, for an arbitrary admissible $u$ we have
    \begin{align}\label{20}
      -\dot p^{\top}x
    \leq q^{\top}x + \textstyle\sum\nolimits_{m=1}^M r_m^{\top}u_m  + p^{\top}(Ax + \textstyle\sum\nolimits_{m=1}^M B_m u_m)
    \end{align}
where the equality holds if and only if 
\begin{align*}
\min_{|u_m(t)| \le e_m x(t)}
\alpha_m(t)u_m(t)
= -|\alpha_m(t)|\,e_m x(t).
\end{align*}
$\alpha_m(t) := r_m^{\top} + p^{\top}(t)B_m$, for all $m=1, \dots M$ and $t\in [0,T]$, equivalently when~\eqref{opt_u_fin} holds. Combining~\eqref{18} and~\eqref{20} 
\begin{align}\label{22}
J(t,x)
&\leq 
 \mathbb{E}^{t,x}\Big[
\int_{t}^{T}
\Big(
q^{\top}x(s)
+ \textstyle\sum\nolimits_{m=1}^M r_m^{\top} u_m(s)
\Big)
ds 
\Big]
\notag\\
&= C\!\left(t,x,u\right).
\end{align}
Since~\eqref{22} holds with equality if and only if~\eqref{opt_u_fin} is satisfied, statement $(ii)$ also follows.
\end{proof}
\section{Discounted Infinite-Horizon $L_1$ Optimal Control Problem}
In this section we consider the problem of minimizing an infinite-horizon discounted expected cost with discount factor $\beta\geq 0$. In particular, the optimal control problem is 
\begin{subequations}\label{optprob_inf}
\begin{align}
    &\min_{u \in \mathcal U(x{(\cdot)})}\mathbb E \big[ \int_0^{\infty} e^{-\beta t} \big( q^{\top}x(t)+\sum_{m=1}^M r^{\top}_m u_m(t) \big) dt \big]\\
    &\hspace{2mm}\text{s.t.}\hspace{1mm} \text{Equation}~\eqref{SDE}, \hspace{1mm} \hspace{1mm} x_0\in \mathbb R^L_+, \hspace{1mm} u(t)\in \mathcal U(x{(\cdot)}).
\end{align}
\end{subequations}
\begin{thm}\label{thm_infty}
    Let the matrices $A, B_m, F_n \in \mathbb R^{L\times L}$ satisfy the properties in Assumption~\ref{ass_1}, and assume that $q \in \mathbb R^L_+$, and $r_m\in \mathbb R^L$, $m=1,\dots, M$. Suppose that~\eqref{ass_nonneg} holds. Then, there exists $p \in \mathbb R^L_+$ such that 
        \begin{align}\label{ARE}
            \beta p &=q+A^{\top}p-\textstyle\sum\nolimits_{m=1}^M e_m|r_m+B^{\top}_m \hspace{0.5mm}p|. 
        \end{align}
        Let $x(t)$ be the state trajectory under the control policy {$u^*_m(t)=-K_m x(t)$ with}
        \begin{align}\label{opt_control_inf}
        {K_m  \in}\;\mathrm{diag}(\mathrm{sign}(r_m^{\top}+p^{\top}B_m))e_m 
        \end{align}
        where $p$ is the vector satisfying \eqref{ARE}, and suppose that \begin{align}\label{ass_discounted}
        \lim_{T\rightarrow \infty} \mathbb E \left[ e^{-\beta T}p^{\top} x(T)\right]=0
    \end{align}
     is satisfied {for every $u\in \mathcal U(x(\cdot))$.} Define $J(x)=p^{\top}x$. Then the following hold: 
    \begin{enumerate}[(i)]
        \item  $J(x)=p^{\top}x$ is the value function of the optimal control problem~\eqref{optprob_inf}.
        \item The optimal control policy is $u_m^*(t)=-K_m x(t)$ with $K_m$ given by~\eqref{opt_control_inf}. 
    \end{enumerate}
\end{thm}
\begin{proof}
    Assume that there exists $p\in \mathbb R^L_+$ such that~\eqref{ARE} and~\eqref{ass_discounted} hold.
    Define $J(x)=p^{\top}x$. Since $J(x)\in C^{1,2}(\mathcal Q)$, the verification theorem~\cite[Lem. 9.1]{viscosity_fleming} applies. Hence, by Dynkin's formula
    \begin{align*}
        &\mathbb E^{x}[e^{-\beta T} J(x(T))]= J(x)+\mathbb E^x \Big[\int_0^Te^{-\beta t}\left( \mathcal G^{u} J(x(t))-\beta J(x(t))\right)dt \Big]
\end{align*}
where the operator $\mathcal G^{u}$ applied to $J$ is 
\begin{align*}
    \mathcal G^{u}J
    &= (Ax + \textstyle\sum\nolimits_{m=1}^M B_m u_m(t))^{\top}
    \partial_x J \\
    &~~~~~~~~~~~~~~~~~~\quad + \tfrac12 \textstyle\sum\nolimits_{n=1}^N 
    \mathrm{Tr}\hspace{0.1mm}\big(F_n x x^{\top}
    F_n^{\top}\partial_x^2 J\big).
\end{align*}
Substituting $J(x)=p^{\top}x$, the second-order term vanishes ($\partial^2_x J=0$) and the stochastic integral has zero expectation. 
\begin{align}\label{27}
        &J(x)=\mathbb E^{x}[e^{-\beta T} J(x(T))]+\mathbb{E}^{x}\Big[
    \int_{0}^{T}
    e^{-\beta t}\Big(p^{\top}\bigl(Ax(t)
    \notag \\
    &~~~~~~~~~~~~~~~~~~~~~~~\qquad+\textstyle\sum\nolimits_{m=1}^M B_m u_m(t)\bigr)
    -\beta p^{\top}x(t)
    \Big)
    \, dt
    \Big].
\end{align} 
It is possible to express the right-hand side of the algebraic equation~\eqref{ARE} as
\begin{align*}
    0= \min_{u}
    \Big(
    q^{\top}x 
    + \sum_{m=1}^M r_m^{\top}u_m
    + p^{\top}(Ax + \sum_m B_mu_m-\beta x)\Big).
\end{align*}
Thus, for any admissible $u$ it is true that
\begin{align}\label{29}
    0\le q^{\top}x + \sum_m r_m^{\top}u_m  + p^{\top}(Ax + \sum_m B_m u_m-\beta x(t))
\end{align}
with equality if and only if 
\begin{align*}
    &\min_{|u_m|\le e_m x} \left( r_m^{\top}+p^{\top}B_m \right)u_m= -|r_m^{\top}+p^{\top}B_m|e_m x.
\end{align*}
for all $m$. Combining~\eqref{27} and~\eqref{29} we obtain
\begin{align*}
    &J(x)\leq \mathbb E^{x}[e^{-\beta T} J(x(T))]\\
    &~~~~~~+ \mathbb{E}^{x}\Big[
    \int_{0}^{T}
    e^{-\beta t}\Big(
    q^{\top}x(t)
    + \textstyle\sum\nolimits_{m=1}^M r_m^{\top} u_m(t)
    \Big)
    dt 
    \Big].
    \notag
\end{align*}
By assumption~\eqref{ass_discounted}, as $T\rightarrow \infty$ yields 
\begin{align*}
   J(x)\le 
\mathbb E \hspace{0.1mm}\Big[\int_{0}^{\infty} e^{-\beta t}
\Big(q^{\top}x(t)+\textstyle\sum\nolimits_{m=1}^M r_m^{\top}u_m(t)\Big)dt\Big]
\end{align*}
for every admissible $u$. Indeed, the equality holds if and only if~\eqref{opt_control_inf} is satisfied. Hence, statement $(ii)$ also follows.
\end{proof}
Theorem~\ref{thm_infty} relies on the transversality condition~\eqref{ass_discounted}. The following remark proposes a sufficient condition that ensures this condition is satisfied.
\begin{rem}
    Define $\bar A=  A+\textstyle\sum\nolimits_{m=1}^M|B_m |e_m$, which is Metzler by Assumption~\ref{ass_1}. For any state trajectory $x(t)\in \mathbb R^{L}_+$ and admissible control satisfying $|u_m(t)|\leq e_m x(t)$, the inequality $ Ax(t)+\sum_{m=1}^M B_mu_m(t) \leq \bar Ax$ holds. Assume $\beta > \alpha (\bar A)$. Under this condition, $\mathbb E[p^{\top}x(t)]\leq C e^{\alpha(\bar A)t}$ for some $C>0$, and hence~\eqref{ass_discounted} holds.
\end{rem}
In the undiscounted formulation, $\beta=0$, the transversality condition becomes $\lim_{{T} \to \infty}\mathbb E[p^{\top}x(T)]=0$. This condition implies that the state trajectory converges to the origin under an admissible control policy. Therefore, in order for Theorem~\ref{thm_infty} to be applicable in the undiscounted setting, the system must admit a stabilizing control law. This motivates introducing the notion of $(e_1,\dots, e_M)$-stabilizability. 
\begin{defn}
    Let $A\in \mathbb R^{L \times L}$, $B_m\in \mathbb R^{L\times L}$ and $e_m\in \mathbb R_+$, $m=1,\dots M$. We say that a tuple $(A,B_1,\dots,B_M)$ is $(e_1,\dots, e_M)$-stabilizable if there exist diagonal feedback gain matrices 
    $K_m\in \mathbb R^{L\times L}$ with $|K_m|\leq I_L e_m$, $m=1,\dots M$ such that $A-B_1 K_1-\dots-B_M K_M$ is Hurwitz.
\end{defn}
\begin{rem}
    To verify the $(e_1,\dots, e_M)$- stabilizability of the tuple $(A, B_1,\dots B_M)$, analogous to~\cite{LR_paperII}, it can be shown that it is necessary and sufficient to verify the feasibility of 
    \begin{align*}
        Ax+\textstyle\sum\nolimits_{m=1}^M B_m u_m \leq -\mathds 1, \hspace{1mm} {- e_mx \leq u_m \leq e_m x}.
    \end{align*}
    This feasibility problem can be verified by any linear program solver.
\end{rem}
The following result shows that the solution of the algebraic equation~\eqref{ARE} is equivalent to the solution of~\eqref{LP}. 
\begin{lem}[Cor. 8 and Thm. 9 ~\cite{LR_paperII}]
    Suppose that~\eqref{ass_nonneg} holds and the tuple $(A, B_1,\dots, B_M)$ is $(e_1,\dots, e_M)$-stabilizable. Then~\eqref{ARE} has a solution $p\geq 0$. The vector $p\geq 0$ solves~\eqref{ARE} if and only if $p$ maximizes linear program
    \begin{align}\label{LP}
        &\mathrm{Maximize} \hspace{2mm} \mathds{1}^{\top}p \hspace{1mm} \mathrm{over} \hspace{1mm} p \in \mathbb{R}^{L}_{+}, \hspace{1mm} \zeta_m \in \mathbb{R}^{L}_{+}, \hspace{1mm} \forall m \notag\\
        &\mathrm{Subject} \hspace{1mm} \mathrm{to} \hspace{2mm} A^{\top}p \geq \textstyle\sum\nolimits_{m=1}^M e_m\zeta_m -q +\beta p  \\
        &\hspace{15mm} -\zeta_m \leq r_m +B_m^{\top}p \leq \zeta_m . \notag
    \end{align} 
\end{lem}
The transversality condition~\eqref{ass_discounted} requires that the first moment of the state converges to zero. Recall that mean square stability implies mean stability but the converse is not true{~\cite{MS_stab}}, this is illustrated in the following example.
\begin{exmp}
    Consider the scalar SDE $dx(t)=ax(t)dt+bx(t)dw(t)$, $x(0)=x_0\in \mathbb R_+$, with parameters satisfying $a<0$ and $2a+b^2>0$. The SDE has the closed-form solution $x(t)=x_0 \mathrm{exp}\big(\big( a-\tfrac{1}{2}b^2\big)t +bw(t)\big).$
    The first and second order moments are $\mathbb E[x(t)]=x_0 e^{at}$, $\mathbb E[x(t)^2]=x_0^2 e^{(2a+b^2)t}$. Since $a<0$, $\mathbb E[x(t)]\to 0$ as $t \to \infty$, so the system is mean stable. However, because $2a+b^2>0$, $\mathbb E[x(t)^2]\to \infty$ as $t \to \infty$. Hence, the system is not mean-square stable.
\end{exmp}
\section{Example: High-Frequency Trading}
\begin{figure*}[t!]
    \centering
    \includegraphics[width=1\textwidth]{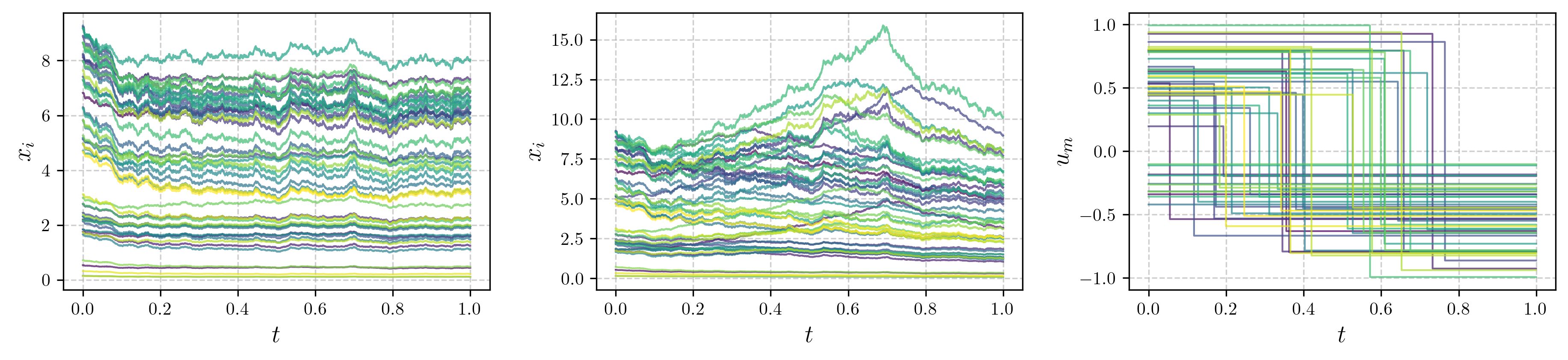}
    \caption{(Left panel) Allocation trajectories under passive control $u =0$. (Middle panel) Optimal allocation trajectories under the switching control. (Right) Optimal control switching policy.}
    \label{exmp_HFT}
\end{figure*} 
We consider a high-frequency trading problem over a one-second horizon, with parameters chosen to reflect a high-volatility, high-friction regime where switching is visible. {The ranges below are illustrative rather than calibrated to real market data, and are chosen to satisfy the assumption~\eqref{ass_nonneg} while keeping the switching behavior visible at the chosen time horizon.} The state $x(t)\in \mathbb R_+^{50}$ represents capital allocated across 50 technology stocks and evolves according to~\eqref{SDE}. The diagonal entries of $A$ model short-term expected returns and are randomly generated in $[-5\times 10^{-4}, \hspace{0.1mm}5\times 10^{-4}]$, while off-diagonal entries in $[0,\hspace{0.1mm} 2\times 10^{-4}]$ capture positive coupling effects. 
Trading is modeled by diagonal matrices $B_m$ ($M=L=50$), so that each control affects only its corresponding stock, i.e. $[B_m]_{ii}=b_m$ if $i=m$ and $0$ otherwise. The gains $b_m$ are generated in $[0.6, \hspace{0.2mm}1.7]$. The control $u_m(t)$  represents the rate at which capital is added to or withdrawn from stock $m$ and the constraint $|u_m(t)|\leq e_m x(t)$, with $e_m\in[0.1, \hspace{0.1mm}1]$, models liquidity limitations. 
Market uncertainty is introduced through two diagonal multiplicative noise channels $F_1$, $F_2$ with entries in $[0.02,\hspace{0.2mm} 0.05]$ and $[0.02, \hspace{0.2mm} 0.08]$, respectively, so that volatility scales 
the position size.
The initial allocation $x_0$ is drawn componentwise in $[0,10]$. We consider the finite-horizon objective~\eqref{max_obj} where $q$ represents the benefit of holding capital 
and $r_m$ models trading costs, satisfying $q > -\sum_{m=1}^{50} |r_m|e_m$ and generated in $[2\times 10^{-4},\hspace{0.1mm} 10^{-3}]$ and $[0.01,\hspace{0.1mm} 0.02]$, respectively. The optimal control has a switching structure indicating when it is optimal to withdraw  capital (risk-off) $u_m(t)=-e_mx(t)$, when increasing capital is advantageous (risk-on), $u_m(t)=e_mx(t)$ and when the marginal value is zero i.e. $u(t)\in [-e_mx(t), \hspace{0.1mm} e_mx(t)]$. 
Figure~\ref{exmp_HFT} illustrates both the allocation trajectories and the switching signals.
For a random parameter choice, the optimal value is $J(0,x_0)=p(0)^{\top}x_0\approx 5.50 $ while the passive strategy $u_m=0$ yields $4.31$, highlighting the performance achieved through dynamic switching at optimal times. Note that $p$ is independent of $F_n$, so $p(0)^{\top}x_0$ is invariant to noise intensity reflecting the robustness in Remark~\ref{rem4}.
\section{Conclusions}
This work presents explicit solutions to an $L_1$ optimal control problem 
for positive systems with multiplicative Gaussian noise and linear input constraints in both finite and discounted infinite-horizon settings. Forward invariance of the positive orthant is established for the associated multivariate Itô diffusion. 
The resulting value function and optimal policy coincide with those of the deterministic formulation, demonstrating robustness to multiplicative stochastic uncertainty. Future work includes identifying other stochastic system dynamics that admit explicit optimal solutions, {exploring alternative noise structures that preserve forward invariance of the positive orthant,} and applications in finance and epidemiology.

\section*{Acknowledgments}

We would like to thank Professor Valery Ugrinovskii for insightful comments and discussions.

\bibliographystyle{abbrv}  
\bibliography{cas-refs}
\end{document}